\documentclass[12pt]{article}
\usepackage{amssymb}
\usepackage{amsthm}
\usepackage{amsmath}
\usepackage{graphicx}

\DeclareMathOperator{\Con}{Con}

\DeclareMathOperator{\T}{\Theta}
\DeclareMathOperator{\PR}{\Phi}
\newtheorem{theorem}{Theorem}%[section]
\newtheorem{definition}[theorem]{Definition}
\newtheorem{lemma}[theorem]{Lemma}
\newtheorem{proposition}[theorem]{Proposition}
\newtheorem{remark}[theorem]{Remark}
\newtheorem{example}[theorem]{Example}
\newtheorem{corollary}[theorem]{Corollary}
\title{Relatively residuated lattices and posets}
\author{Ivan Chajda and Helmut L\"anger}
\date{}
\begin{document}
\footnotetext[1]{Support of the research by \"OAD, project CZ~02/2019, and support of the research of the first author by IGA, project P\v rF~2019~015, is gratefully acknowledged.}
\maketitle
\begin{abstract}
It is known that every relatively pseudocomplemented lattice is residuated and, moreover, it is distributive. Unfortunately, non-distributive lattices with a unary operation satisfying properties similar to relative pseudocomplementation cannot be converted in residuated ones. The aim of our paper is to introduce a more general concept of a relative residuated lattice in such a way that also non-modular sectionally pseudocomplemented lattices are included. We derive several properties of relative residuated lattices which are similar to those known for residuated ones and extend our results to posets.
 \end{abstract}
 


The history of residuated lattices goes back to Dilworth in 1939, see e.g.\ \cite D. He generalized the situation known for relative pseudocomplemented lattices by replacing meet by a general binary operation $\odot$ and the operation of relative pseudocomplementation by a general binary operation $\rightarrow$. The usefulness of this approach found its precipitation in a number of papers and monographs devoted to residuated lattices. Nowadays this theory serves as an algebraic semantics of several kinds of substructural logic, in particular of fuzzy logics. In this context we refer to the monographs \cite B and \cite{JT} and the survey \cite{GJKO}.

Unfortunately, not every lattice equipped with a unary operation can be converted into a residuated one. The authors showed recently that if adjointness is replaced by left adjointness then every orthomodular lattice can be organized into a left residuated one, see \cite{CL17}. The first aim of this paper is to show that if adjointness is relativized to certain intervals of a given lattice then e.g.\ every sectionally pseudocomplemented lattice can be converted into such a relatively residuated lattice. It is well-known (see e.g.\ \cite D) that every relatively pseudocomplemented lattice is distributive. However, our sectionally pseudocomplemented lattices even need not be modular as shown below. Hence, we extended residuation also to this case.

The natural question arises if this concept can be generalized also to posets. It was shown recently by the authors (see \cite{CL18} and \cite{CLP}) that in some cases this is possible, in particular for relatively pseudocomplemented posets, Boolean posets or pseudo-orthomodular posets. Since we present results on sectionally pseudocomplemented lattices, we try to generalize our concepts also to sectionally pseudocomplemented posets and we show that every such poset can be organized into a so-called relatively operator residuated one.

Among other things this shows that also sectionally pseudocomplemented lattices can be considered as algebraic semantics of certain substructural logics and similarly also sectionally pseudocomplemented posets in the case when the logic in question need not have a defined disjunction.

Recall that a {\em lattice} $(L,\vee,\wedge,)$ is called {\em relatively pseudocomplemented} if for every $a,b\in L$ there exists a greatest element $x$ of $L$ satisfying $a\wedge x\leq b$. This element $x$ is called the {\em relative pseudocomplement of $a$ with respect to $b$}.

\begin{definition}
{\rm(}cf.\ {\rm\cite C)} Let $(L,\vee,\wedge)$ be a lattice and $a,b,d\in L$. Then $d$ is called the {\em sectional pseudocomplement of $a$ with respect to $b$} {\rm(}denoted by $a*b${\rm)} if
\[
d\text{ is the greatest element }x\text{ of }L\text{ satisfying }(a\vee b)\wedge x=b.
\]
Of course, any such $x$ must be in $[b,1]$. There exists at least one such $x$, namely $x=b$. A {\em sectionally pseudocomplemented lattice} is an algebra $\mathbf L=(L,\vee,\wedge,*)$ of type $(2,2,2)$ such that $(L,\vee,\wedge)$ is a lattice and for all $x,y\in L$, $x*y$ is the sectional pseudocomplement of $x$ with respect to $y$.
\end{definition}

Obviously, every relatively pseudocomplemented lattice is also sectionally pseudocomplemented because the relative pseudocomplement $(a\vee b)*b$ is in fact the sectional pseudocomplement of $a$ with respect to $b$. However, there exist sectionally  pseudocomplemented lattices which are not relatively pseudocomplemented.

It is well known that every relatively pseudocomplemented lattice is distributive. The advantage of sectionally pseudocomplemented lattices is that there exist also non-modular ones, see the following example taken from \cite C.

\begin{example}
The lattice $\mathbf N_5$ visualized in Fig.~1

\vspace*{-2mm}

\[
\setlength{\unitlength}{7mm}
\begin{picture}(6,9)
\put(3,2){\circle*{.3}}
\put(5,4){\circle*{.3}}
\put(1,5){\circle*{.3}}
\put(5,6){\circle*{.3}}
\put(3,8){\circle*{.3}}
\put(3,2){\line(-2,3)2}
\put(3,2){\line(1,1)2}
\put(5,4){\line(0,1)2}
\put(3,8){\line(-2,-3)2}
\put(3,8){\line(1,-1)2}
\put(2.875,1.25){$0$}
\put(5.4,3.85){$a$}
\put(.35,4.85){$b$}
\put(5.4,5.85){$c$}
\put(2.85,8.4){$1$}
\put(2.2,.3){{\rm Fig.~1}}
\end{picture}
\]

\vspace*{-3mm}

is sectionally pseudocomplemented but not relatively pseudocomplemented because the relative pseudocomplement of $c$ with respect to $a$ does not exist. The operation table for $*$ looks as follows:
\[
\begin{array}{c|ccccc}
* & 0 & a & b & c & 1 \\
\hline
0 & 1 & 1 & 1 & 1 & 1 \\
a & b & 1 & b & 1 & 1 \\
b & c & a & 1 & c & 1 \\
c & b & a & b & 1 & 1 \\
1 & 0 & a & b & c & 1
\end{array}
\]
\end{example}

\begin{remark}
If $(L,\vee,\wedge,*,1)$ is a sectionally pseudocomplemented lattice and $a,b\in L$ then $a\leq b$ if and only if $a*b=1$.
\end{remark}

\begin{definition}\label{def1}
A {\em relatively residuated lattice} is an algebra $\mathbf L=(L,\vee,\wedge,\odot,\rightarrow,1)$ of type $(2,2,$ $2,2,0)$ such that $(L,\vee,\wedge,1)$ is a lattice with $1$ and for all $a,b,c\in L$ the following conditions hold:
\begin{itemize}
\item $(L,\odot,1)$ is a commutative groupoid with neutral element,
\item $a\leq b$ implies $a\odot c\leq b\odot c$,
\item $(a\vee b)\odot(c\vee b)\leq b$ if and only if $c\vee b\leq a\rightarrow b$.
\end{itemize}
The last condition will be called {\em relative adjointness}. $\mathbf L$ is called {\em divisible} if it satisfies the identity $(x\vee y)\odot(x\rightarrow y)\approx y$.
\end{definition}

\begin{theorem}\label{th1}
Let $(L,\vee,\wedge,*,1)$ be an algebra of type $(2,2,2,0)$ such that $(L,\vee,$ $\wedge,1)$ is a lattice with $1$. Then $(L,\vee,\wedge,*)$ is a sectionally pseudocomplemented lattice if and only if $(L,\vee,\wedge,\wedge,*,$ $1)$ is a divisible relatively residuated lattice.
\end{theorem}

\begin{proof}
For all $a,b\in L$ the following are equivalent:
\begin{align*}
& a*b\text{ is the sectional pseudocomplement of }a\text{ with respect to }b, \\
& \text{for all }c\in L\text{ we have }(a\vee b)\wedge(c\vee b)\leq b\text{ if and only if }c\vee b\leq a*b.
\end{align*}
If this is the case then $(a\vee b)\wedge(a*b)\leq b$ which together with $a*b\geq b$ yields $(a\vee b)\wedge(a*b)=b$ proving divisibility.
\end{proof}

It is worth noticing that there are relatively residuated lattices where $\odot$ does not coincide with $\wedge$ and $\rightarrow$ is not the sectional pseudocomplement. The next example shows such a case. Although this lattice is even sectionally pseudocomplemented, we define $\odot$ and $\rightarrow$ in a different way.

\begin{example}\label{ex1}
If $(L,\vee,\wedge)=(\{0,a,1\},\vee,\wedge)$ denotes the three-element lattice and $\odot$ and $\rightarrow$ are defined by
\[
\begin{array}{c|ccc}
\odot & 0 & a & 1 \\
\hline
  0   & 0 & 0 & 0 \\
	a   & 0 & 0 & a \\
	1   & 0 & a & 1
\end{array}
\quad\quad\quad
\begin{array}{c|ccc}
\rightarrow & 0 & a & 1 \\
\hline
     0      & 1 & 1 & 1 \\
		 a      & a & 1 & 1 \\
		 1      & 0 & a & 1
\end{array}
\]
then $(L,\vee,\wedge,\odot,\rightarrow,1)$ is a relatively residuated lattice which is not divisible since
\[
(a\vee0)\wedge(a\rightarrow0)=a\wedge a=a\neq0.
\]
\end{example}

Recall that a lattice $(L,\vee,\wedge)$ is called {\em meet-semidistributive} if $a,b,c\in L$ and $a\wedge b=a\wedge c$ together imply $a\wedge(b\vee c)=a\wedge b$.

The following result follows by Theorem~1 in \cite C and Theorem~\ref{th1}.

\begin{corollary}
Let $(L,\vee,\wedge,1)$ be a finite lattice with $1$. Then the following are equivalent:
\begin{enumerate}
\item[{\rm(i)}] There exists a binary operation $*$ on $L$ such that $(L,\vee,\wedge,\wedge,*,1)$ is relatively residuated,
\item[{\rm(ii)}] $(L,\vee,\wedge)$ is meet-semidistributive.
\end{enumerate}
\end{corollary}

\begin{proof}
Let $a,b,c\in L$. \\
(i) $\Rightarrow$ (ii): \\
Assume $a\wedge b=a\wedge c$. Because of $(a\vee(a\wedge b))\wedge(b\vee(a\wedge b))\leq a\wedge b$ we have $b=b\vee(a\wedge b)\leq a*(a\wedge b)$ according to relative adjointness. Analogously we obtain $c\leq a*(a\wedge c)=a*(a\wedge b)$. Hence $(b\vee c)\vee(a\wedge b)=b\vee c\leq a*(a\wedge b)$ whence $a\wedge(b\vee c)=(a\vee(a\wedge b))\wedge((b\vee c)\vee(a\wedge b))\leq a\wedge b$ again according to relative adjointness, i.e.\ $a\wedge(b\vee c)=a\wedge b$. \\
(ii) $\Rightarrow$ (i): \\
We define $a*b:=\bigvee\{x\in[b,1]\mid(a\vee b)\wedge x=b\}$. Since $(a\vee b)\wedge b=b$ we have $b\leq a*b$ and since $(L,\vee,\wedge)$ is meet-semidistributive we have $(a\vee b)\wedge(a*b)=b$. Finally, if $b\leq c$ and $(a\vee b)\wedge c=b$ then $c\leq a*b$. This shows that $(L,\vee,\wedge,*,1)$ is sectionally pseudocomplemented.
\end{proof}

The next theorem lists several important properties of relative residuated lattices showing essential similarities with residuated lattices.

\begin{theorem}\label{th2}
Let $\mathbf L=(L,\vee,\wedge,\odot,\rightarrow,1)$ be a relatively residuated lattice and $a,b,c\in L$. Then the following hold:
\begin{enumerate}
\item[{\rm(i)}] $1\rightarrow x\approx x$,
\item[{\rm(ii)}] $a\leq b$ if and only if $a\rightarrow b=1$,
\item[{\rm(iii)}] $a\odot(a\vee b)\leq a$,
\item[{\rm(iv)}] $b\leq a\rightarrow b$,
\item[{\rm(v)}] $(a\vee b)\odot(a\rightarrow b)\leq b$,
\item[{\rm(vi)}] $x\rightarrow y\approx(x\vee y)\rightarrow y$,
\item[{\rm(vii)}] $a\vee b\leq(a\rightarrow b)\rightarrow b$,
\item[{\rm(viii)}] $a\leq b$ implies $b\rightarrow c\leq a\rightarrow c$,
\item[{\rm(ix)}] if $\mathbf L$ has a $0$ then $a\odot b=0$ if and only if $a\leq b\rightarrow0$ and hence $0\odot x\approx0$.
\end{enumerate}
\end{theorem}

\begin{proof}
\
\begin{enumerate}
\item[(i)] The following are equivalent:
\begin{align*}
                                     a & \leq a, \\
               (1\vee a)\odot(a\vee a) & \leq a, \\
                               a\vee a & \leq1\rightarrow a, \\
                                     a & \leq1\rightarrow a, \\
                (1\rightarrow a)\vee a & \leq1\rightarrow a, \\
(1\vee a)\odot((1\rightarrow a)\vee a) & \leq a, \\
                        1\rightarrow a & \leq a.
\end{align*}
\item[(ii)] The following are equivalent:
\begin{align*}
                      a & \leq b, \\
                a\vee b & \leq b, \\
(a\vee b)\odot(1\vee b) & \leq b, \\
                1\vee b & \leq a\rightarrow b, \\
         a\rightarrow b & =1.
\end{align*}
\item[(iii)] The following are equivalent:
\begin{align*}
                b\vee a & \leq1, \\
                b\vee a & \leq a\rightarrow a, \\
(a\vee a)\odot(b\vee a) & \leq a, \\
        a\odot(a\vee b) & \leq a.
\end{align*}
\item[(iv)] The following are equivalent:
\begin{align*}
        b\odot(a\vee b) & \leq b, \\
       (a\vee b)\odot b & \leq b, \\
(a\vee b)\odot(b\vee b) & \leq b, \\
                b\vee b & \leq a\rightarrow b, \\
                      b & \leq a\rightarrow b.
\end{align*}
\item[(v)] The following are equivalent:
\begin{align*}
                        a\rightarrow b & \leq a\rightarrow b, \\
                (a\rightarrow b)\vee b & \leq a\rightarrow b, \\
(a\vee b)\odot((a\rightarrow b)\vee b) & \leq b, \\
        (a\vee b)\odot(a\rightarrow b) & \leq b.
\end{align*}
\item[(vi)] The following are equivalent:
\begin{align*}
                (a\vee b)\odot(a\rightarrow b) & \leq b, \\
((a\vee b)\vee b)\odot((a\rightarrow b)\vee b) & \leq b, \\
                        (a\rightarrow b)\vee b & \leq(a\vee b)\rightarrow b, \\
                                a\rightarrow b & \leq(a\vee b)\rightarrow b.
\end{align*}
Conversely, the following are equivalent:
\begin{align*}
((a\vee b)\vee b)\odot((a\vee b)\rightarrow b) & \leq b, \\
(a\vee b)\odot(((a\vee b)\rightarrow b)\vee b) & \leq b, \\
((a\vee b)\rightarrow b)\vee b & \leq a\rightarrow b, \\
(a\vee b)\rightarrow b & \leq a\rightarrow b.
\end{align*}
\item[(vii)] The following are equivalent:
\begin{align*}
        (a\vee b)\odot(a\rightarrow b) & \leq b, \\
        (a\rightarrow b)\odot(a\vee b) & \leq b, \\
((a\rightarrow b)\vee b)\odot(a\vee b) & \leq b, \\
                               a\vee b & \leq(a\rightarrow b)\rightarrow b.
\end{align*}
\item[(viii)] Everyone of the following assertions implies the next one:
\begin{align*}
        (b\vee c)\odot(b\rightarrow c) & \leq c, \\
(a\vee c)\odot((b\rightarrow c)\vee c) & \leq c, \\
                (b\rightarrow c)\vee c & \leq a\rightarrow c, \\
                        b\rightarrow c & \leq a\rightarrow c.
\end{align*}
\item[(ix)] If $\mathbf L$ has a $0$ then the following are equivalent:
\begin{align*}
             a\odot b & =0, \\
             b\odot a & =0, \\
(b\vee0)\odot(a\vee0) & \leq0, \\
               a\vee0 & \leq b\rightarrow0, \\
                    a & \leq b\rightarrow0.
\end{align*}
\end{enumerate}
\end{proof}

Next we prove that relatively residuated lattices satisfy rather strong congruence properties similarly to residuated lattices.

\begin{theorem}\label{th4}
Every relatively residuated lattice is arithmetical, i.e.\ congruence per- \\
mutable and congruence distributive.
\end{theorem}

\begin{proof}
Let $\mathbf L=(L,\vee,\wedge,\odot,\rightarrow,1)$ be a relatively residuated lattice, $a,b,c\in A$ and $\Theta,\Phi\in\Con\mathbf L$. We use (i), (ii) and (vii) of Theorem~\ref{th2}. If $(a,c)\in\Theta\circ\Phi$ then there exists some $b\in L$ with $a\T b\PR c$ and hence
\begin{align*}
a & =((a\rightarrow c)\rightarrow c)\wedge a=((a\rightarrow c)\rightarrow c)\wedge((c\rightarrow c)\rightarrow a)\PR \\
& \PR((a\rightarrow b)\rightarrow c)\wedge((c\rightarrow b)\rightarrow a)\T((a\rightarrow a)\rightarrow c)\wedge((c\rightarrow a)\rightarrow a)= \\
& =c\wedge((c\rightarrow a)\rightarrow a)=c
\end{align*}
showing $(a,c)\in\Phi\circ\Theta$. Hence $\Theta\circ\Phi\subseteq\Phi\circ\Theta$. Since $\Theta$ and $\Phi$ were arbitrary congruences on $\mathbf L$, we obtain $\Theta\circ\Phi=\Phi\circ\Theta$. Congruence distributivity of $\mathbf L$ follows since $(L,\vee,\wedge)$ is a lattice.
\end{proof}

We are going to show that if $\mathbf L$ is a lattice with two additional binary operations $\odot$ and $\rightarrow$ where $\odot$ is monotone and condition (v) of Theorem~\ref{th2} is satisfied then $\mathbf L$ satisfies one implication of relative adjointness.

\begin{lemma}\label{lem1}
Let $(L,\vee,\wedge)$ be a lattice and $\odot$ and $\rightarrow$ binary operations on $L$ such that for all $a,b,c\in L$ the following conditions hold:
\begin{itemize}
\item $b\leq c$ implies $a\odot b\leq a\odot c$,
\item $(a\vee b)\odot(a\rightarrow b)\leq b$.
\end{itemize}
Then for all $a,b,c\in L$ the following holds:
\begin{itemize}
\item $c\vee b\leq a\rightarrow b$ implies $(a\vee b)\odot(c\vee b)\leq b$.
\end{itemize}
\end{lemma}

\begin{proof}
If $a,b,c\in L$ and $c\vee b\leq a\rightarrow b$ then $(a\vee b)\odot(c\vee b)\leq(a\vee b)\odot(a\rightarrow b)\leq b$.
\end{proof}

Although residuated lattices satisfy the identity
\[
(x\odot y)\rightarrow z\approx x\rightarrow(y\rightarrow z)
\]
for commutative $\odot$ (see e.g.\ \cite B), this need not hold in the relatively residuated case. However, if we assume the condition
\[
((a\vee b)\odot(c\vee b))\rightarrow b\leq(c\vee b)\rightarrow(a\rightarrow b)
\]
(which can be re-written as an identity) then we are able to prove also the converse implication of relative adjointness.

\begin{theorem}
Let $\mathcal V$ denote the variety of algebras $(L,\vee,\wedge,\odot,\rightarrow,1)$ of type $(2,2,2,2,0)$ satisfying the identities of lattices with $1$, the identities of commutative groupoids with $1$ and the following conditions {\rm(}which can be re-written as identities{\rm)} for all $a,b,c\in L$:
\begin{enumerate}
\item[{\rm(i)}] $((a\vee b)\odot(c\vee b))\rightarrow b\leq(c\vee b)\rightarrow(a\rightarrow b)$,
\item[{\rm(ii)}] $(a\vee b)\odot(a\rightarrow b)\leq b$,
\item[{\rm(iii)}] $a\odot b\leq a\odot(b\vee c)$,
\item[{\rm(iv)}] $x\rightarrow(x\vee y)\approx1$.
\end{enumerate}
Then $\mathcal V$ is a variety of relative residuated lattices.
\end{theorem}

\begin{proof}
Let $a,b,c\in L$. By Lemma~\ref{lem1}, every member of $\mathcal V$ satisfies one implication of relative adjointness. We prove the converse implication. If $a\rightarrow b=1$ then, using (ii), we infer $a\leq a\vee b=(a\vee b)\odot1=(a\vee b)\odot(a\rightarrow b)\leq b$. Conversely, $a\leq b$ implies $a\rightarrow b=1$ according to (iv). Together,
\[
a\leq b\text{ if and only if }a\rightarrow b=1.
\]
Now, if $(a\vee b)\odot(c\vee b)\leq b$ then $((a\vee b)\odot(c\vee b))\rightarrow b=1$ and, by (i), also $(c\vee b)\rightarrow(a\rightarrow b)=1$ whence $c\vee b\leq a\rightarrow b$. Together, $\mathbf L$ satisfies relative adjointness. The remaining conditions of Definition~\ref{def1} are evident. Thus $\mathbf L$ is a relatively residuated lattice.
\end{proof}

We can show that $\mathcal V$ satisfies one more congruence property than those mentioned in Theorem~\ref{th4}, namely weak regularity. This property expresses the fact that every congruence on some member of $\mathcal V$ is fully determined by its kernel. The precise definition of this notion is as follows.

An {\em algebra} $\mathbf A$ having a constant $1$ is called {\em weakly regular} if for any $\Theta,\Phi\in\Con\mathbf A$ we have that  $[1]\Theta=[1]\Phi$ implies $\Theta=\Phi$. A {\em variety} is called {\em weakly regular} if every of its members has this property.

\begin{theorem}
The variety $\mathcal V$ is arithmetical and weakly regular.
\end{theorem}

\begin{proof}
That $\mathcal V$ is arithmetical, follows from Theorem~\ref{th4}. Weak regularity of $\mathcal V$ is equivalent to the fact that there exists some positive integer $n$ and binary terms $t_1(x,y),\ldots,t_n(x,y)$ such that $t_1(x,y)=\cdots=t_n(x,y)=1$ is equivalent to $x=y$ (cf.\ Theorem~6.4.3 in \cite{CEL}). If we put $n:=2$, $t_1(x,y):=x\rightarrow y$ and $t_2(x,y):=y\rightarrow x$ then according to (ii) of Theorem~\ref{th2} this condition is satisfied.
\end{proof}

Now, we want to extend our previous investigations concerning  lattices to ordered sets.

\begin{definition}\label{def2}
Let $(P,\leq)$ be a poset and $a,b,d\in P$. Then $d$ is called the {\em sectional pseudocomplement of $a$ with respect to $b$} {\rm(}denoted by $a*b${\rm)} if for all $c\in P$,
\[
L(U(a,b),U(c,b))=L(b)\text{ if and only if }d\in U(c,b).
\]
It will be shown that if such an element $d$ exists then it is unique and $d\geq b$. A {\em sectionally pseudocomplemented poset} is an ordered triple $(P,\leq,*)$ such that $(P,\leq)$ is a poset and for all $x,y\in P$, $x*y$ is the sectional pseudocomplement of $x$ with respect to $y$.
\end{definition}

\begin{lemma}
Let $(P,\leq)$ be a poset and $a,b,d\in P$ and assume $d$ to satisfy the condition of Definition~\ref{def2}. Then $d$ is unique, $d\geq b$ and $L(U(a,b),d)=L(b)$.
\end{lemma}

\begin{proof}
The following are equivalent:
\begin{align*}
L(U(a,b),U(b,b)) & =L(b), \\
               d & \in U(b,b), \\
               d & \in U(d,b), \\
L(U(a,b),U(d,b)) & =L(b).
\end{align*}
This shows
\[
U(d)=\bigcap\{U(c,b)\mid c\in P,L(U(a,b),U(c,b))=L(b)\}.
\]
Hence $U(d)$ and therefore also $d$ is unique and, moreover $d\geq b$ and $L(U(a,b),d)=L(b)$.
\end{proof}

We are going to show that if a sectionally pseudocomplemented lattice is considered as a poset then it is surely a sectionally pseudocomplemented poset and, moreover, the sectional pseudocomplements coincide. Hence, Definition~\ref{def2} is sound.

\begin{lemma}
Let $\mathbf L=(L,\vee,\wedge)$ be a lattice, $\mathbf P:=(L,\leq)$ and $a,b\in L$. Then $a*b$ exists in $\mathbf L$ if and only if $a*b$ exists in $\mathbf P$ and in this case they are equal.
\end{lemma}

\begin{proof}
First assume $a*b$ to exist in $\mathbf L$. Then for all $c\in L$ the following are equivalent:
\begin{align*}
        L(U(a,b),U(c,b)) & =L(b), \\
(a\vee b)\wedge(c\vee b) & =b, \\
                 c\vee b & \leq a*b, \\
                     a*b & \in U(c,b).
\end{align*}
Hence $a*b$ is the sectional pseudocomplement of $a$ with respect to $b$ in $\mathbf P$. Conversely, assume $a*b$ to exist in $\mathbf P$. Then the following are equivalent:
\begin{align*}
                 a*b & \in U(a*b,b), \\
  L(U(a,b),U(a*b,b)) & =L(b), \\
(a\vee b)\wedge(a*b) & =b.
\end{align*}
Moreover, for every $c\in L$ any of the following assertions implies the next one:
\begin{align*}
(a\vee b)\wedge c & =b, \\
 L(U(a,b),U(c,b)) & =L(b), \\
              a*b & \in U(c,b), \\
                c & \leq a*b.
\end{align*}
This shows that $a*b$ is the sectional pseudocomplement of $a$ with respect to $b$ in $\mathbf L$.
\end{proof}

\begin{example}
The poset $\mathbf P_6$ visualized in Fig.~2

\vspace*{-2mm}

\[
\setlength{\unitlength}{7mm}
\begin{picture}(6,9)
\put(3,2){\circle*{.3}}
\put(1,4){\circle*{.3}}
\put(5,4){\circle*{.3}}
\put(1,6){\circle*{.3}}
\put(5,6){\circle*{.3}}
\put(3,8){\circle*{.3}}
\put(3,2){\line(-1,1)2}
\put(3,2){\line(1,1)2}
\put(1,6){\line(0,-1)2}
\put(1,6){\line(2,-1)4}
\put(1,6){\line(1,1)2}
\put(5,6){\line(0,-1)2}
\put(5,6){\line(-2,-1)4}
\put(5,6){\line(-1,1)2}
\put(2.875,1.25){$0$}
\put(.35,3.85){$a$}
\put(5.4,3.85){$b$}
\put(.35,5.85){$c$}
\put(5.4,5.85){$d$}
\put(2.85,8.4){$1$}
\put(2.2,.3){{\rm Fig.~2}}
\end{picture}
\]

\vspace*{-3mm}

is sectionally pseudocomplemented and the operation table for $*$ looks as follows:
\[
\begin{array}{c|cccccc}
* & 0 & a & b & c & d & 1 \\
\hline
0 & 1 & 1 & 1 & 1 & 1 & 1 \\
a & b & 1 & b & 1 & 1 & 1 \\
b & a & a & 1 & 1 & 1 & 1 \\
c & 0 & a & b & 1 & d & 1 \\
d & 0 & a & b & c & 1 & 1 \\
1 & 0 & a & b & c & d & 1
\end{array}
\]
Unfortunately, the poset $\mathbf P_6$ is also relatively pseudocomplemented. In order to obtain a sectionally pseudocomplemented poset which is neither relatively pseudocomplemented nor a lattice we can take the direct product of $\mathbf P_6$ and $\mathbf N_5$. In $\mathbf N_5$ the relative pseudocomplement of $c$ with respect to $a$ does not exist whereas the sectional pseudocomplement of $c$ with respect to $a$ equals $a$.
\end{example}

The definition of a relatively residuated lattice can be modified for poset in the following way:

\begin{definition}
A {\em relatively operator residuated poset} is an ordered quintuple $(P,\leq,M,R,1)$ such that $(P,\leq,1)$ is a poset with $1$, $M$ is a binary operation on $2^P$, $R$ is a mapping from $P^2$ to $2^P$ and for all $a,b,c\in P$ and all $A,B\subseteq P$ the following conditions hold :
\begin{itemize}
\item $M(A,B)\approx M(B,A)$,
\item $M(1,A)\approx M(A,1)\approx L(A)$,
\item $M(U(a,b),U(c,b))\subseteq L(b)$ if and only if $LU(c,b)\subseteq R(a,b)$.
\end{itemize}
The last condition will be called {\em operator relative adjointness}. {\rm(}Here and in the following we will write $M(a,A)$ and $M(A,a)$ instead of $M(\{a\},A)$ and $M(A,\{a\})$, respectively.{\rm)} 
\end{definition}

Similarly as for lattices, we can state and prove the following.

\begin{theorem}\label{th3}
Let $(P,\leq,1)$ be a poset with $1$ and $*$ a binary operation on $P$ and put
\begin{align*}
M(A,B) & :=L(A,B), \\
R(x,y) & :=L(x*y)
\end{align*}
for all $x,y\in P$ and all $A,B\subseteq P$. Then $(P,\leq,*)$ is a pseudocomplemented poset if and only if $(P,\leq,M,R,1)$ is a relatively operator residuated poset.
\end{theorem}

\begin{proof}
We have
\begin{align*}
M(A,B) & \approx L(A,B)\approx L(B,A)\approx M(B,A), \\
M(1,A) & \approx M(A,1)\approx L(1,A)\approx L(A)
\end{align*}
and for all $a,b,c\in P$, $L(U(a,b),U(c,b))=L(b)$ is equivalent to $M(U(a,b),U(c,b))\subseteq L(b)$ and $a*b\in U(c,b)$ is equivalent to $LU(c,b)\subseteq R(a*b)$.
\end{proof}

The next result shows some properties of relatively operator residuated posets analogous to that of Theorem~\ref{th2} for lattices.

\begin{proposition}\label{prop1}
Let $\mathbf P=(P,\leq,M,R,1)$ be a relatively operator residuated poset and $a,b\in P$. Then the following hold:
\begin{enumerate}
\item[{\rm(i)}] $L(a)\subseteq R(1,a)$,
\item[{\rm(ii)}] $a\leq b$ if and only if $R(a,b)=P$,
\item[{\rm(iii)}] $M(U(a),U(a,b))\subseteq L(a)$,
\item[{\rm(iv)}] $L(b)\subseteq R(a,b)$,
\item[{\rm(v)}] if $\mathbf P$ has a $0$ then $M(U(a),U(b))\subseteq\{0\}$ if and only if $L(a)\subseteq R(b,0)$.
\end{enumerate}
\end{proposition}

\begin{proof}
\
\begin{enumerate}
\item[(i)] The following are equivalent:
\begin{align*}
            L(a) & \subseteq L(a), \\
M(U(1,a),U(a,a)) & \subseteq L(a), \\
         LU(a,a) & \subseteq R(1,a), \\
            L(a) & \subseteq R(1,a).
\end{align*}
\item[(ii)] The following are equivalent:
\begin{align*}
               a & \leq b, \\
          U(a,b) & \supseteq U(b), \\
         LU(a,b) & \subseteq L(b), \\
M(U(a,b),U(1,b)) & \subseteq L(b), \\
         LU(1,b) & \subseteq R(a,b), \\
          R(a,b) & =P.
\end{align*}
\item[(iii)] The following are equivalent:
\begin{align*}
         LU(b,a) & \subseteq R(a,a), \\
M(U(a,a),U(b,a)) & \subseteq L(a), \\
  M(U(a),U(a,b)) & \subseteq L(a).
\end{align*}
\item[(iv)] The following are equivalent:
\begin{align*}
  M(U(b),U(b,a)) & \subseteq L(b), \\
M(U(a,b),U(b,b)) & \subseteq L(b), \\
         LU(b,b) & \subseteq R(a,b), \\
            L(b) & \subseteq R(a,b).
\end{align*}
\item[(v)] If $\mathbf P$ has a $0$ then the following are equivalent:
\begin{align*}
    M(U(a),U(b)) & \subseteq\{0\}, \\
M(U(b,0),U(a,0)) & \subseteq L(0), \\
         LU(a,0) & \subseteq R(b,0), \\
            L(a) & \subseteq R(b,0).
\end{align*}
\end{enumerate}
\end{proof}

\begin{corollary}
If $(P,\leq,*,1)$ is a sectionally pseudocomplemented poset and $a,b\in P$ then $a\leq b$ if and only if $a*b=1$.
\end{corollary}

\begin{proof}
This follows from Theorem~\ref{th3} and (ii) of Proposition~\ref{prop1}.
\end{proof}

Authors' addresses:

Ivan Chajda \\
Palack\'y University Olomouc \\
Faculty of Science \\
Department of Algebra and Geometry \\
17.\ listopadu 12 \\
771 46 Olomouc \\
Czech Republic \\
ivan.chajda@upol.cz

Helmut L\"anger \\
TU Wien \\
Faculty of Mathematics and Geoinformation \\
Institute of Discrete Mathematics and Geometry \\
Wiedner Hauptstra\ss e 8-10 \\
1040 Vienna \\
Austria, and \\
Palack\'y University Olomouc \\
Faculty of Science \\
Department of Algebra and Geometry \\
17.\ listopadu 12 \\
771 46 Olomouc \\
Czech Republic \\
helmut.laenger@tuwien.ac.at
\end{document}